\documentclass[leqno]{article}
\usepackage{amsmath,amsfonts,amssymb,amsthm}
\usepackage[letterpaper,top=2cm,bottom=2cm,left=3cm,right=3cm, marginparwidth=2cm, 
]{geometry}

\usepackage{graphicx}
\usepackage{dcolumn}
\usepackage{bm}
\usepackage{amsthm}
\newtheorem{theorem}{Theorem}

\usepackage{hyperref}
\hypersetup{
	colorlinks=true,       
	linkcolor=blue,          
	citecolor=blue,        
	filecolor=magenta,      
	urlcolor=blue           
}
\begin{document}
{\centering\LARGE Generation of Multi-Scroll Attractors Without Equilibria Via Piecewise Linear Systems\par}
\bigskip
{\centering\large R.J.~Escalante-Gonz\'alez$^1$,  E.~Campos-Cant\'on$^2$ \par \footnotetext[2]{Corresponding Author}}
{\centering\itshape Divisi\'on de Matem\'aticas Aplicadas, Instituto Potosino de Investigaci\'on Cient\'ifica y Tecnol\'ogica A. C., Camino a la Presa San Jos\'e 2055, Col. Lomas 4 Secci\'on, C.P. 78216, San Luis Potos\'i, S.L.P., M\'exico. $^1$rodolfo.escalante@ipicyt.edu.mx, $^2$eric.campos@ipicyt.edu.mx\par}
\bigskip
{\centering \large Matthew Nicol\par}
{\centering\itshape Mathematics Department, University of Houston, \par Houston, Texas 77204-3008, USA.\par nicol@math.uh.edu\par}
\bigskip
\begin{abstract}
	In this paper we present a new class of dynamical system without equilibria which possesses a multi scroll attractor. It is a piecewise-linear (PWL) system which is simple, stable, displays chaotic behavior and serves as a model for analogous non-linear systems. We test for chaos using the 0-1 Test for Chaos of Ref.~\cite{Gottwald}.
\end{abstract}

\section{Introduction}
In recent years  the study of dynamical systems with complicated dynamics but without equilibria has attracted attention.
Since the first dynamical system of this kind with a chaotic attractor was  introduced in Ref.~\cite{Sprott94} (Sprott case A), several works have investigated this topic. Several three-dimensional (3D) autonomous dynamical systems which exhibit chaotic attractors and whose associated vector fields present quadratic nonlinearities have been reported, for example Ref.~\cite{Wang13} and the 17 NE systems given in Ref.~\cite{Sajad13}. Also, four-dimensional (4D) autonomous systems have been exhibited, such  as the four-wing non-equilibrium chaotic system Ref.~\cite{LIN16} and the systems with hyperchaotic attractors (equivalently two positive Lyapunov exponents) in Ref.~\cite{Wang12-2,Pham16}, whose vector fields have quadratic and  cubic nonlinearities.

These studies have not been restricted to integer orders, in~\cite{Cafagna14} a new fractional-order chaotic system without equilibrium points was presented, this system represents the fractional-order counterpart of the integer-order system NE$_6$ studied in Ref.~\cite{Sajad13}.

The first piecewise linear system without equilibrium points with an hyperchaotic attractor was reported in~\cite{Li12}, this system was based on the diffusion-less Lorenz system by approximating  the quadratic nonlinearities with the sign and absolute value functions.
Recently, in Ref.~\cite{Rodolfo}, a class of PWL dynamical systems without equilibria was reported, the attractors exhibited by these systems can be easily shifted along the $x$ axis by  displacing the switching plane. In fact, the attractors generated by these systems without equilibria fulfill the definition of hidden attractor given in Ref.~\cite{Dudkowski16}.

In this paper we present a new class of PWL dynamical system without equilibria that generates a multiscroll attractor. It  is based on linear affine transformations of the form $A\boldsymbol{x}+B$ where $A$ is a  $3\times 3$ matrix with two complex conjugate eigenvalues with positive real part and the third eigenvalue either zero, slightly negative or positive.
The construction for $A$ having a zero eigenvalue is presented in section~\ref{sec:multiscrollconst} and the case when no eigenvalue of $A$ is zero due to a perturbation is presented in section~\ref{sec:multiscrollconst2}, numerical simulations are provided for both cases.

\section{\label{sec:multiscrollconst}PWL dynamical system with multiscroll attractor:  singular matrix case.}

We will first present the construction where $A$  has two complex conjugate eigenvalues with positive real part and the  third eigenvalue is zero. In the next section we will study the effect of allowing the 
zero eigenvalue to become slightly negative $-\epsilon$ or slightly positive $\epsilon$ for  a parameter $\epsilon >0$. 
In order to introduce the new class of PWL dynamical system we first review some useful theorems used in the construction. The next theorem gives necessary and sufficient conditions for the absence of equilibria in a system based on a linear affine transformation. These systems will be used as sub-systems for the PWL construction.

\begin{theorem}\cite{Rodolfo}
	Given a dynamical system based on an affine transformation of the form $\boldsymbol{\dot{x}}=A\boldsymbol{x}+B$ where
	$\boldsymbol{x}\in \mathbb{R}^n$ is the state vector, $B\in \mathbb{R}^n$ is a nonzero constant vector and $A \in \mathbb{R}^{n\times n}$ is a linear operator, the system possesses no equilibrium point if and only if:
	\begin{itemize}
		\item $A$ is not invertible and 
		\item $B$ is linearly independent of the set of non-zero vectors comprised by columns of the operator $A$.
	\end{itemize}
	\label{th:noequi}
\end{theorem}

The next theorem tells us, for a specific type of linear operator, which vectors are linearly independent of the non-zero column vectors of the associated matrix of the operator which will help us to fulfill the conditions in Theorem~\ref{th:noequi}.

\begin{theorem}
	Suppose $A\in\mathbb{R}^{3\times 3}$ is a linear operator whose characteristic polynomial has zero as a  single root (so that zero has  algebraic multiplicity one). Then any eigenvector associated to zero is linearly independent of the column vectors of the matrix $A$.
	\label{th:indep}
\end{theorem}
\begin{proof}
	By considering the Jordan Canonical Form there is an invariant subspace $U$ of  dimension two  corresponding to the other non-zero eigenvalues (either two real eigenvalues or a 
	complex conjugate pair of eigenvalues $a+ib$, $a-ib$ with $b\not =0$) and a one dimensional invariant subspace $V$
	corresponding to the eigenvalue $\lambda=0$.  Thus  $U\oplus V=\mathbb{R}^3$ and $A$ restricted to $U$ is invertible, hence $AU\subset U$. If $x\in \mathbb{R}^3$ then we may write (uniquely) $x=x_u\oplus x_v$
	where $x_u\in U$ and $x_v \in V$. Then $A x= Ax_u + A x_v=Ax_u$ as $Ax_v=0$. Since $Ax_u\subset U$ we cannot solve $Ax=w$ where $w\in V$ is  a nonzero vector.\end{proof}

We will illustrate our mechanism for producing multi-scroll chaotic attractors   via a detailed  example. 

\subsection{Example.}

Consider a dynamical system whose associated vector field is the linear system:
\begin{equation}\label{ec:linear}
\boldsymbol{\dot{x}}=A\boldsymbol{ x},
\end{equation}
where $\boldsymbol{x} \in \mathbb{R}^3$ is the state vector, and $A\in\mathbb{R}^{3\times 3}$ is a linear operator. The matrix $A$ has eigenvalues $\lambda_i$, $i=1,2,3$ where $\lambda_1$, $\lambda_2$ are complex conjugate eigenvalues with positive real part while $\lambda_3=0$.\\

To illustrate this class of dynamical systems with multiscroll attractors consider the model system with a vector field of the form (\ref{ec:linear}) in $\mathbb{R}^3$ whose matrix  $A$ is given as follows:

\begin{equation}\label{eq:Amatrix1}
A=\begin{bmatrix}
m  & -n & 0\\
n 	 & m & 0\\
0 & 0  & 0
\end{bmatrix},\hspace{.3cm}
A=[a_1,a_2,a_3],
\end{equation}
where $a_1$, $a_2$ and $a_3$ are the column vectors of the matrix $A$ and we suppose $m>0$ and $n\neq0$. The eigenvectors $V$ associated to the eigenvalue $\lambda=0$ are given as follows:

\begin{equation}
V=(0,0,v)^T,
\end{equation}
with $v\neq0$.\\
Note that $A$ is of rank two and its column space equals the two-dimensional unstable subspace $<a_1,a_2>$.  We will now consider
the vector field formed by  adding a vector $k_1 a_1+k_2 a_2$ in the span of the  column vectors of the linear operator $A$. 
\begin{equation}
\boldsymbol{\dot{x}}=A\boldsymbol{x}+k_1a_1+k_2a_2.
\end{equation}
Using the matrix $A$ given by \eqref{eq:Amatrix1}, we have the following linear system:
\begin{equation}\label{eq:Aks}
\boldsymbol{\dot{x}}=
\begin{bmatrix}
m& -n & 0\\
n& m & 0\\
0& 0 & 0
\end{bmatrix}
\begin{bmatrix}
x_1+k_1\\x_2+k_2\\x_3
\end{bmatrix}.
\end{equation}
Considering the following change of variables $y_1=x_1+k_1$, $y_2=x_2+k_2$ and $y_3=x_3$, the system \eqref{eq:Aks} is given as follows:
\begin{equation}\label{eq:Ay}
\boldsymbol{\dot{y}}=
\begin{bmatrix}
m& -n & 0\\
n& m & 0\\
0& 0 & 0
\end{bmatrix}
\begin{bmatrix}
y_1\\y_2\\y_3
\end{bmatrix}=A\boldsymbol{y},
\end{equation}
so that the rotation in  the unstable subspace is around the axis given by the line 
\[
\begin{bmatrix}
- k_1\\- k_2\\0
\end{bmatrix}
+
t\begin{bmatrix}
0\\0\\1
\end{bmatrix}.
\]
If  a non-zero vector $V$ in the neutral direction is added then the point  \[
\begin{bmatrix}
- k_1\\- k_2\\0
\end{bmatrix},
\]
is no longer an  equilibrium and the vector field at this point is equal to 
\[
V= \begin{bmatrix}
0\\0\\v
\end{bmatrix},
\]
so that
\begin{equation}\label{eq:AfinneLS}
\boldsymbol{\dot{y}}=A\boldsymbol{y}+V.
\end{equation}

The solution of the initial value problem $X_0$ for \eqref{eq:AfinneLS} is given by
\begin{equation}\label{eq:solAy}
\boldsymbol{y(t)}=
\begin{bmatrix}
e^{mt}cos(nt)& -e^{mt}sin(nt) & 0\\
e^{mt}sin(nt)& e^{mt}cos(nt) & 0\\
0& 0 & 0
\end{bmatrix}
\begin{bmatrix}
y_1(0)\\y_2(0)\\y_3(0)
\end{bmatrix}+
\begin{bmatrix}
0\\0\\v*t
\end{bmatrix},\end{equation}

\begin{equation}\label{ec:solAx}
\boldsymbol{x(t)}=
e^{mt}
\begin{bmatrix}
cos(nt)& -sin(nt) & 0\\
sin(nt)& cos(nt) & 0\\
0& 0 & 0
\end{bmatrix}
\begin{bmatrix}
x_1(0)+k_1\\x_2(0)+k_2\\x_3(0)
\end{bmatrix}+
\begin{bmatrix}
-k_1\\-k_2\\v*t
\end{bmatrix}.\end{equation}

\begin{figure}[h!]
	\centering
	\includegraphics[width=0.7\columnwidth]{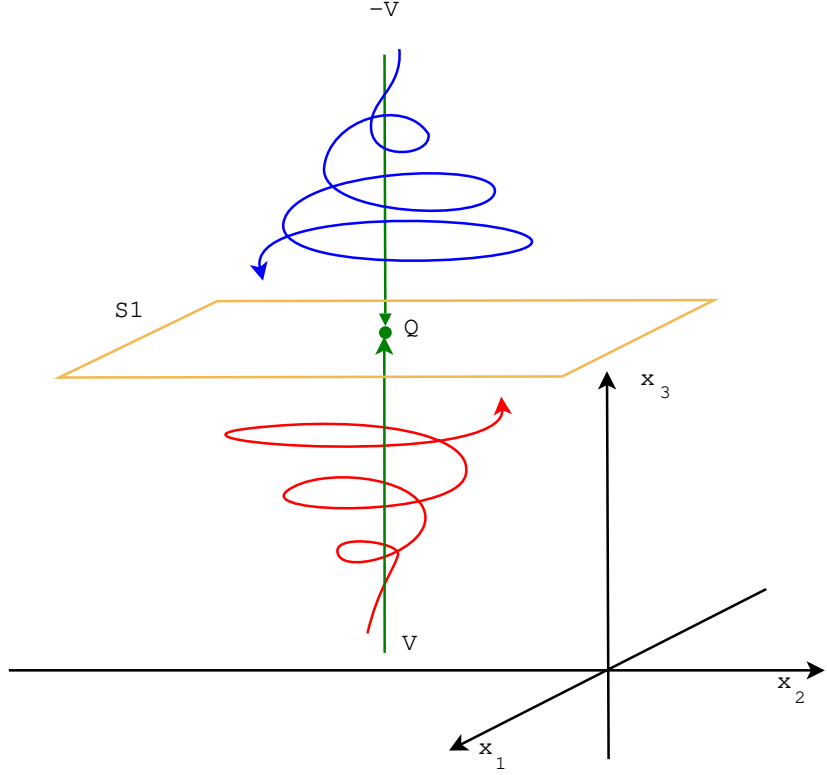}
	\caption{\label{fig:Saddle}A ``saddle-focus like'' locate at the point $Q=(-k_1,-k_2,\tau)$.}
\end{figure}

We wish  to generate a ``saddle-focus like" behavior which bounds motion, so a piecewise linear system is constructed  as follows. A switching surface $S$ will be a hyperplane oriented by its positive normal.  $S$ divides $\mathbb{R}^3$ into two connected components, if a point $\boldsymbol{x}$ lies in the component pointed to by the positive normal of $S$ we will write $\boldsymbol{x}>S$.

To illustrate we  take our switching surface to be $x_3=0$ and define a flow on 
$S$ by 
\begin{equation}\label{eq:HyperPWLS}
\boldsymbol{\dot{x}}=
\begin{cases}
A\boldsymbol{x}+ W, & \text{if } x_1 <0;\\
A\boldsymbol{x}-W, & \text{if } x_1 \geq 0;
\end{cases}
\end{equation}
The vector $W$ is chosen in the plane $S$ (hence in the unstable subspace of $A$), with sign such  that the $x_1$ component of $p_1:=-A_{|S}^{-1}W$ is positive. This ensures that the flow
on $S\cap \{x_1 <0\}$ has an unstable focus at $p_1$, while the flow on $S\cap \{x_1 \ge 0\}$ has an unstable focus at $-p_1$. Thus there is no equilibria for the PWL flow on $S$. Next we define the flow on $\boldsymbol{x}>S$ and $\boldsymbol{x}<S$. 
\begin{equation}\label{eq:2PWLS}
\boldsymbol{\dot{x}}=
\begin{cases}
A\boldsymbol{x}+V+W, & \text{if } \boldsymbol{x} <S, x_1< 0;\\
A\boldsymbol{x}+V-W, & \text{if } \boldsymbol{x} <S, x_1 \ge 0;\\
A\boldsymbol{x}-V+W, & \text{if } \boldsymbol{x} > S, x_1< 0 ;\\
A\boldsymbol{x}-V-W, & \text{if } \boldsymbol{x} > S; x_1 \ge 0;
\end{cases}
\end{equation}
where  we choose the direction of $V$ so that the dynamical system defined by \eqref{eq:2PWLS} $\phi:\mathbb{R}\times \mathbb{R}^3 \to \mathbb{R}^3$
gives the motion of the points towards  the switching plane
$S$.

See ~Figure~\ref{fig:Saddle} where $W=k_1a_1+k_2a_2$ and the plane $S$ is given by $x_3=\tau$.
With a pair of systems $F_1(x)$, $F_3 (x)$ defined respectively  by parallel switching surfaces $S_1$, $S_3$ similar to those described by \eqref{eq:HyperPWLS} and \eqref{eq:2PWLS} along with another switching surface $S_2$ transverse to the other two   it is possible to generate multiscroll attractors.
For example a double scroll-attractor can be generated with the following Piecewise Linear System provided that $F_3(\boldsymbol{x})$ is correctly displaced:
\begin{equation}\label{eq:4PWLS}
\boldsymbol{\dot{x}}=
\begin{cases}
F_1(\boldsymbol{x}), & \text{if } \boldsymbol{x}<S_2;\\
F_3(\boldsymbol{x}), & \text{if } \boldsymbol{x} \geq S_2.
\end{cases}
\end{equation}
Note that switching surfaces $S_1$ and $S_3$ are used to generate two   ``saddle-focus like" behavior, and switching surface $S_2$
is used to commute between these two ``saddle-focus like" behavior, see Fig.~\ref{fig:2spirals}.
The surface $S_2$ is responsible for the stretching and folding behavior in the system in order to generate chaos. This switching surface is constrained to be transverse to the unstable manifold and neutral manifold. Note that the planes $S_1$, $S_2$ and $S_3$ are not invariant under the flow defined by \eqref{eq:4PWLS}. 

It is possible to define a system with 3-scroll attractors, for example as follows (see Fig.~\ref{fig:3spirals}): 
\begin{equation}\label{eq:6PWLS}
\boldsymbol{\dot{x}}=
\begin{cases}
F_1(\boldsymbol{x}), & \text{if } \boldsymbol{x}<S_2;\\
F_2(\boldsymbol{x}), & \text{if } S_2 \le \boldsymbol{x}<S_4;\\
F_3(\boldsymbol{x}), & \text{if } \boldsymbol{x} \geq S_4.
\end{cases}
\end{equation}
Such a system demonstrates the chaotic nature of  the associated flow via symbolic dynamics. If we take small $\epsilon$-neighborhoods of the surfaces $S_1$, $S_3$ and $S_5$ and 
record a $1$, $3$ or $5$ each time a trajectory under the flow crosses into that neighborhood from outside that neighborhood we generate a shift on the space $\{1,3,5\}^N$. Both the symbol sequences
$135$, $131$ and correspondingly $535$ and $531$ may occur, and the symbol sequences generated are complex, non-periodic and we call them ``chaotic''. How disordered the system is in terms of entropy for example, is a subject for another paper.
\begin{figure}[h!]
	\centering
	\includegraphics[width=0.7\columnwidth]{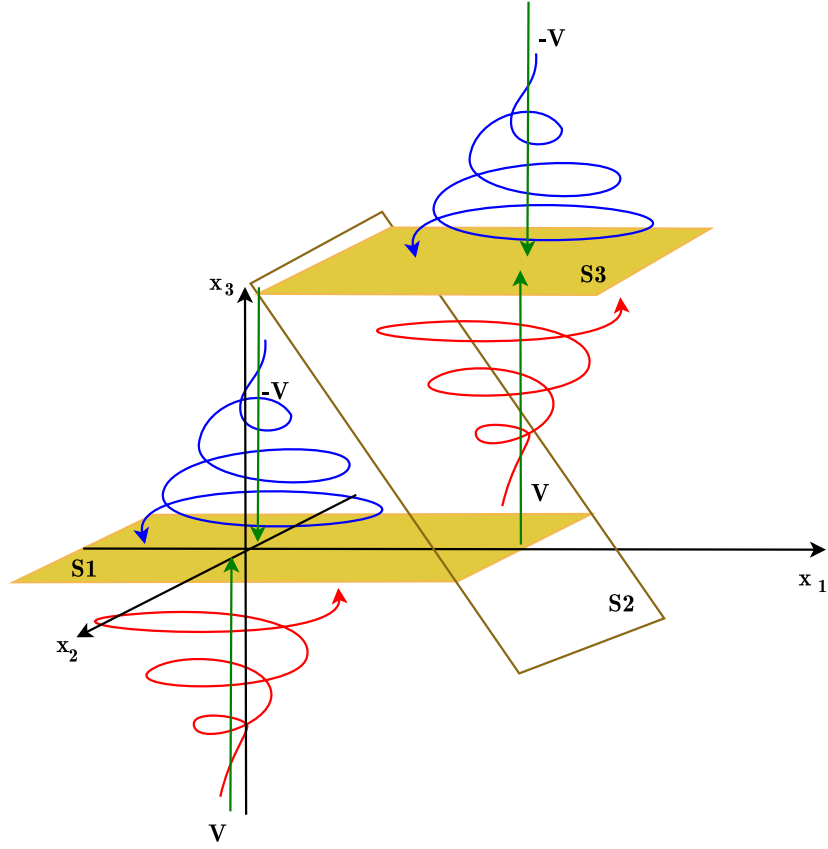}
	\caption{\label{fig:2spirals}A mechanism to generate a double-scroll chaotic attractors without equilibria.}
\end{figure}
\begin{figure}[h!]
	\centering
	\includegraphics[width=0.7\columnwidth]{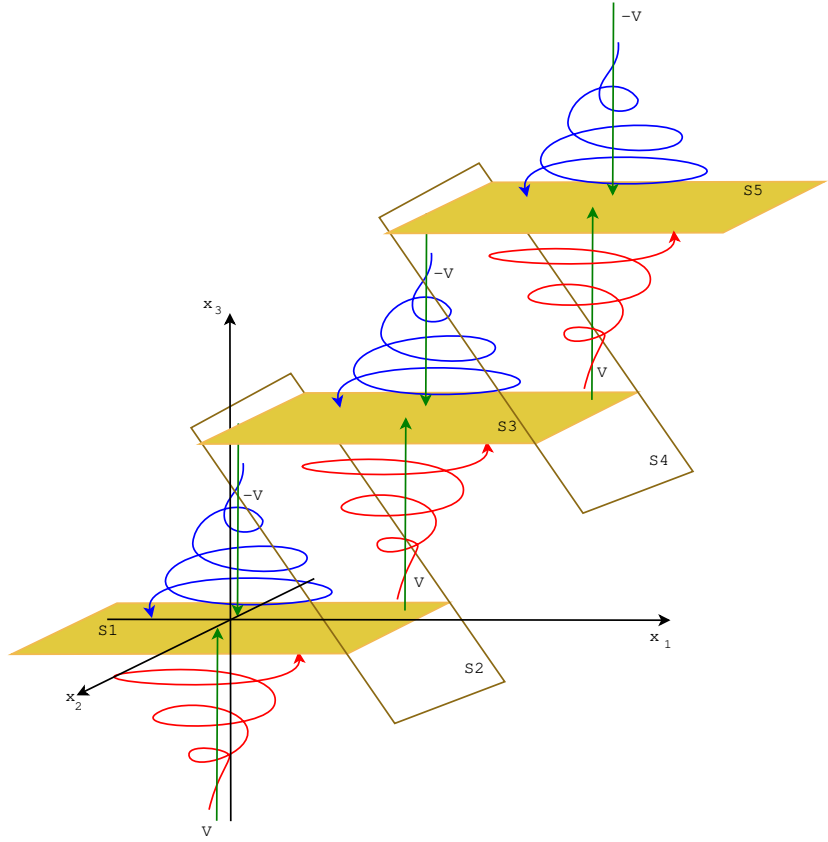}
	\caption{\label{fig:3spirals}A mechanism to generate a triple-scroll  chaotic attractors without equilibria.}
\end{figure}

\subsection{Numerical simulations}

{\bf Example 1.} Figure~\ref{fig:3Dattractor} shows a double scroll attractor which was obtained by using 4th order Runge Kutta (0.01 integration step)    and
considering the following matrix $A$ and vector $V$.
\begin{equation}\label{eq:Amatrixex1}
A=
\begin{bmatrix}
0.5& -10 & 0\\
10& 0.5 & 0\\
0& 0 & 0
\end{bmatrix},
V=\begin{bmatrix}
0\\
0\\
5
\end{bmatrix}.
\end{equation}
with PWL system (details can be seen in Appendix I):
\begin{equation}\label{eq:2SPWLS}
\boldsymbol{\dot{x}}=F_i({\boldsymbol{x}}), i=1,\ldots , 12.
\end{equation}

The switching surface are given by the planes $S_1: x_3=0$, $S_2: x_1+x_3/2=1$ and $S_3: x_3=2$ .
The vectors $W_i$, with $i=1,\ldots,4$ are given in Table~\ref{tab:example1} (see  Appendix I)\\


The resulting attractor with a double scroll is shown in the Figure~\ref{fig:3Dattractor}~(a) and its projections on the planes $(x_1,x_2)$, $(x_1,x_3)$ and $(x_2,x_3)$ in the Figures~\ref{fig:3Dattractor}~(b), \ref{fig:3Dattractor}~(c) and \ref{fig:3Dattractor}~(d), respectively. Its largest Lyapunov exponent calculated is $\lambda=0.97$ (Figure~\ref{fig:3Dattractor}~(e)).\\

\begin{figure}[h!]
	\centering
	(a)
	\includegraphics[width=0.4\columnwidth]{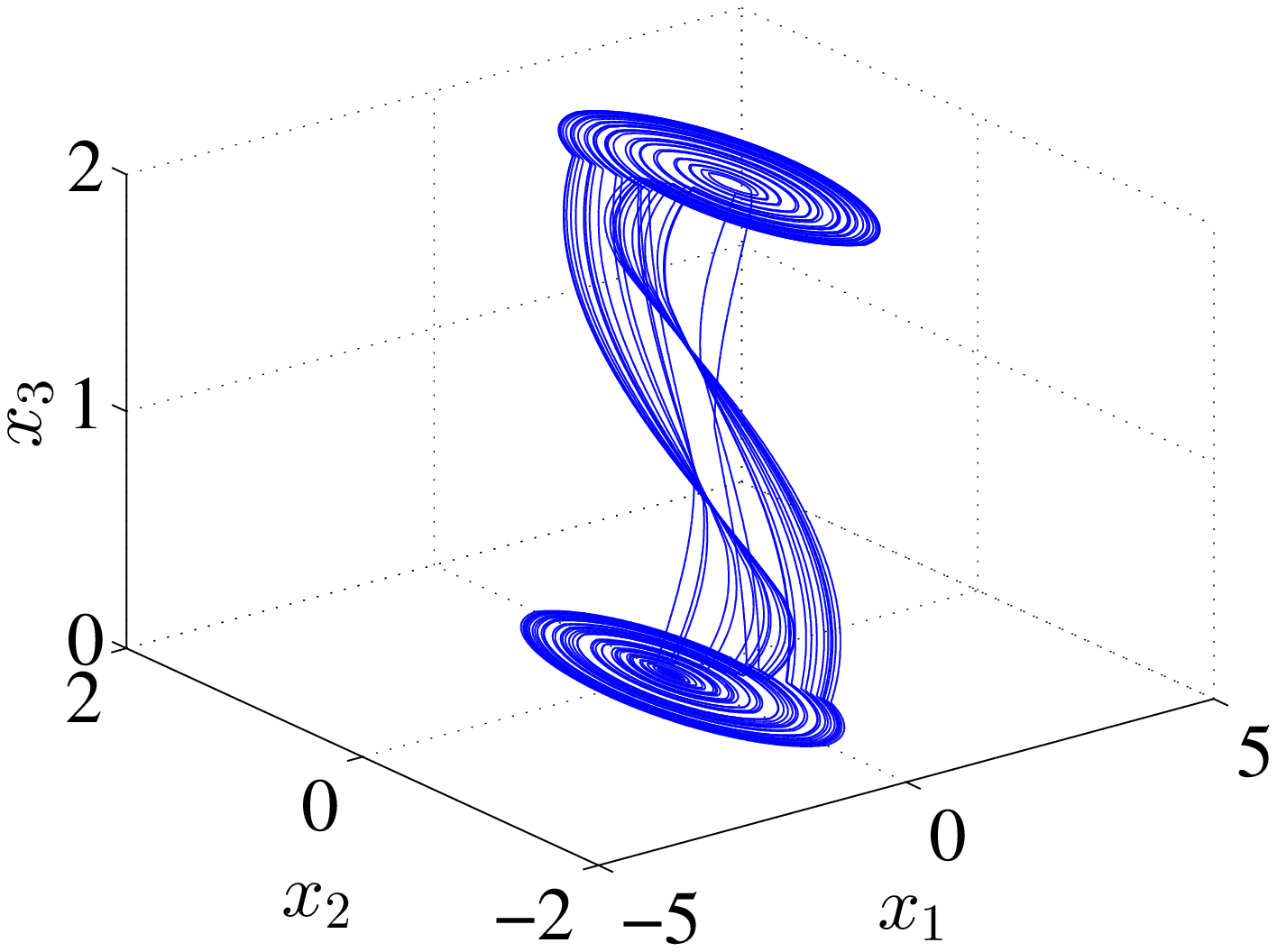}
	(b)
	\includegraphics[width=0.4\columnwidth]{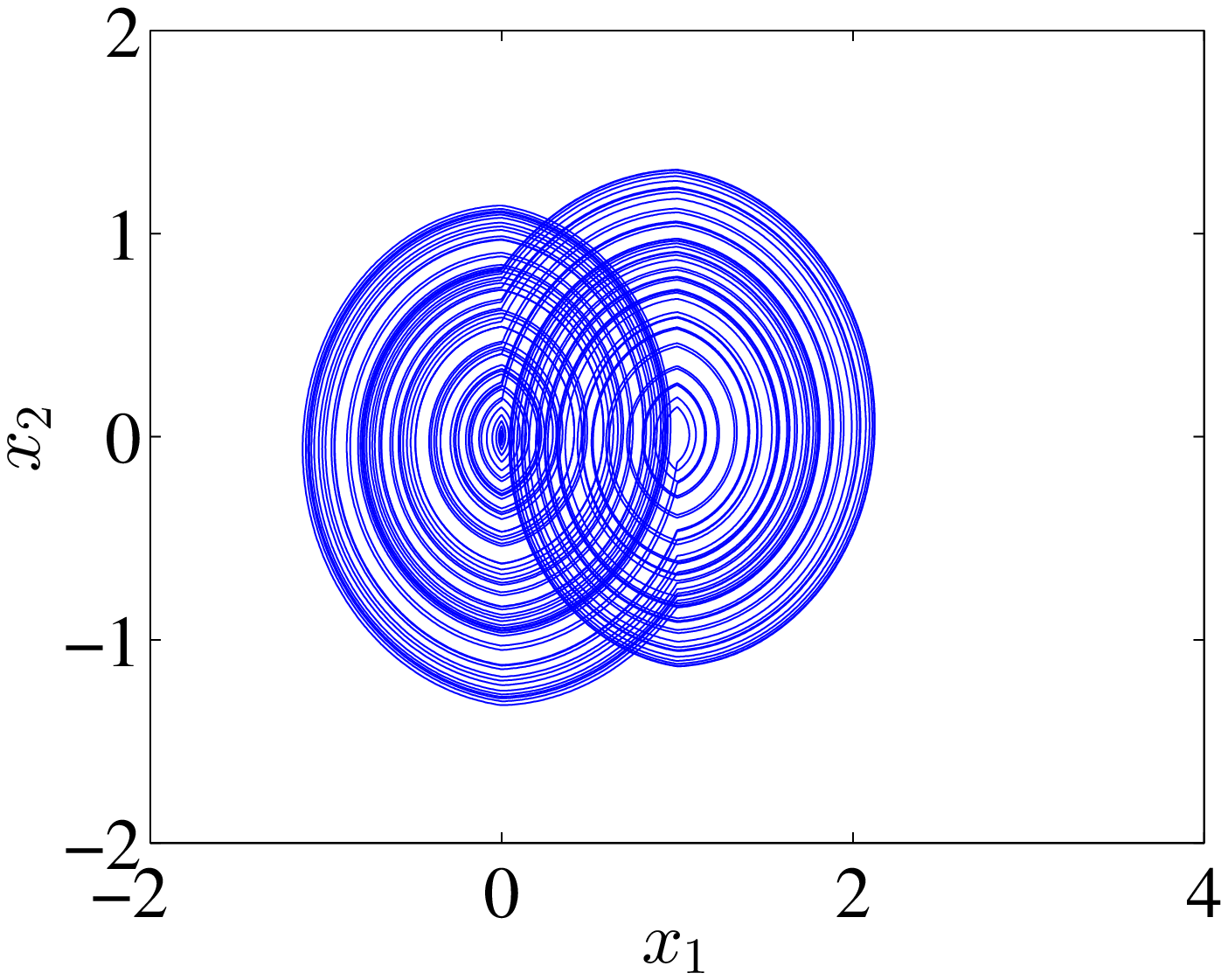}
	\\
	(c)
	\includegraphics[width=0.4\columnwidth]{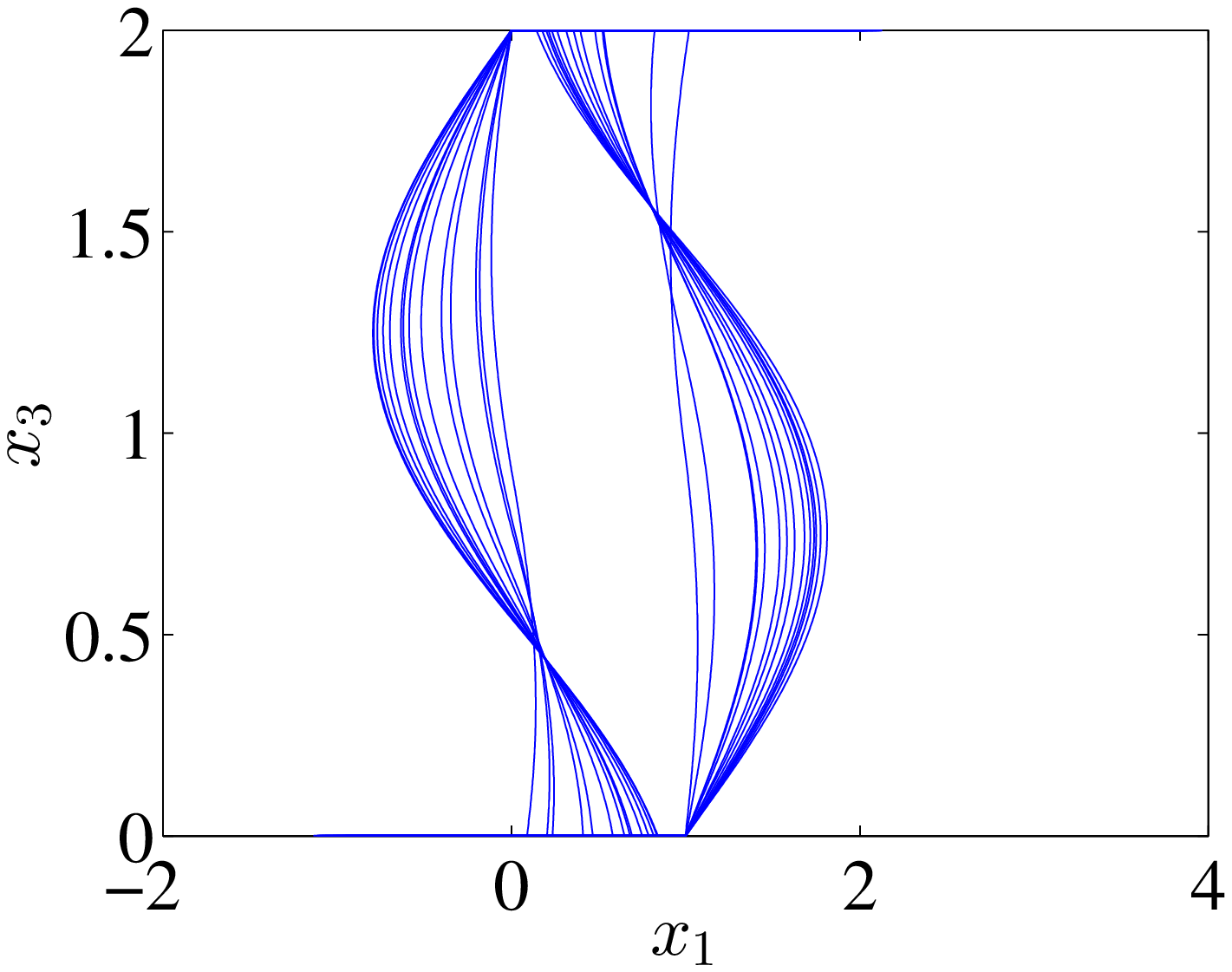}
	(d)
	\includegraphics[width=0.4\columnwidth]{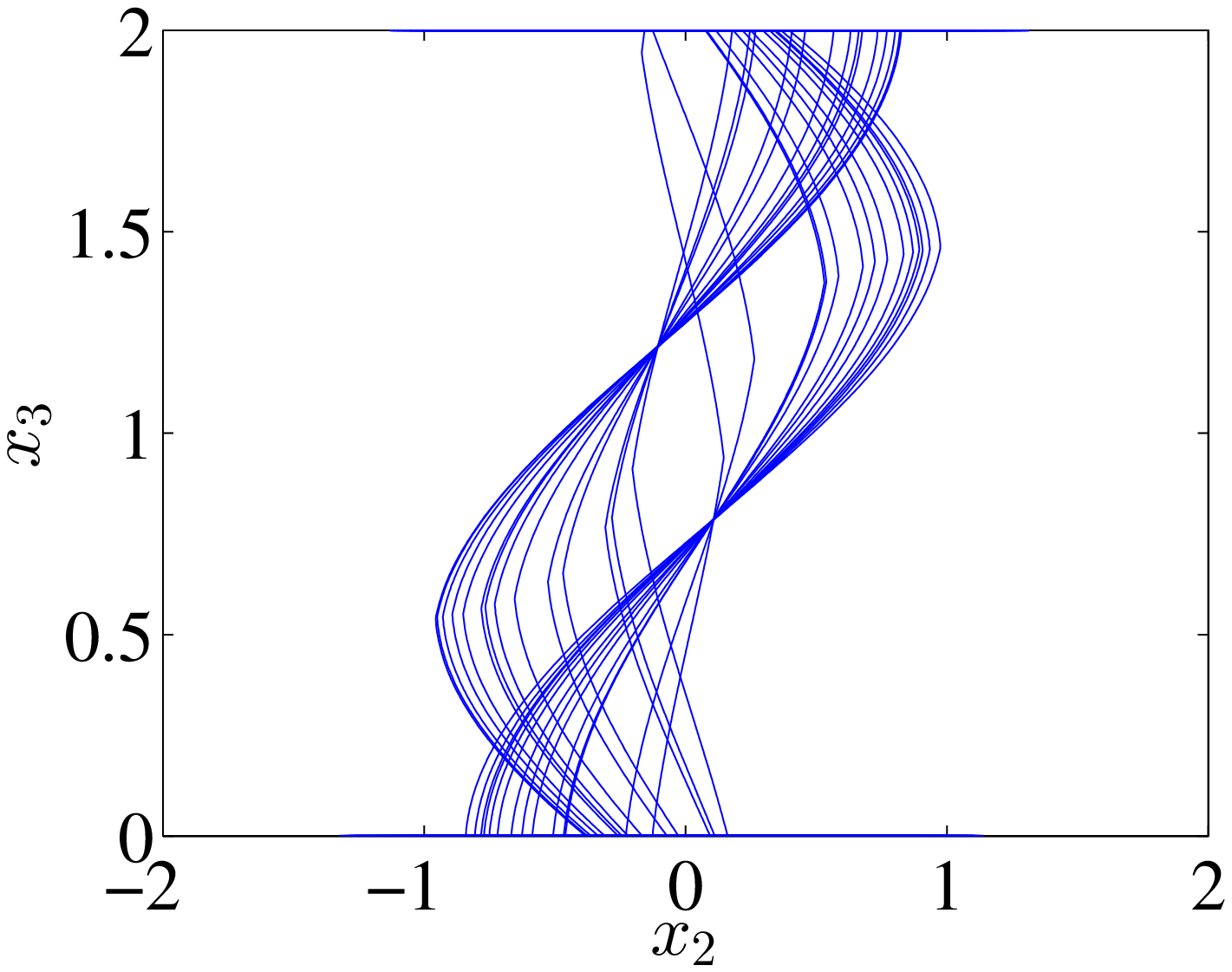}\\
	(e)
	\includegraphics[width=0.4\columnwidth]{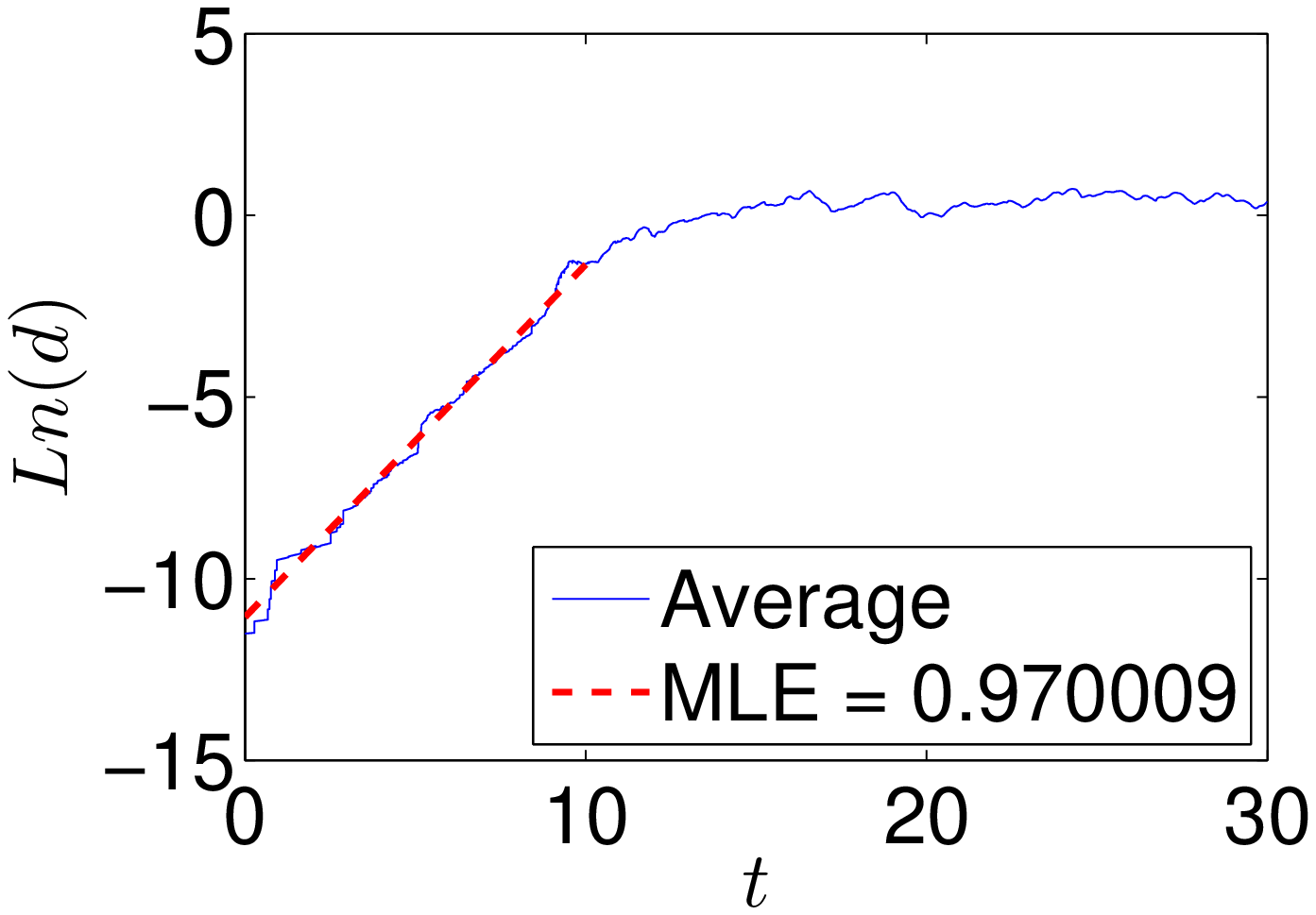}
	\caption{\label{fig:3Dattractor}Attractor of the system (\ref{eq:2SPWLS}) with $A$ and $V$ given in (\ref{eq:Amatrixex1}) for the initial condition $(0,0,0)$ in (a) the space $(x_1,x_2,x_3)$ and its projections onto the planes: (b) $(x_1,x_2)$, (c) $(x_1,x_3)$ and (d) $(x_2,x_3)$. In (e) the largest Lyapunov exponent.}
\end{figure}

It can be extended to a triple scroll attractor by considering the system given by \eqref{eq:6PWLS} with the aditional surfaces $S_4: x_1+x_3/2=3$ and $S_5: x_3=4$ and $W_i$, with $i=1,\ldots,6$ given in Table~\ref{tab:example1}  (see  Appendix I).\\

The resulting attractor with a triple scroll is shown in the Figure~\ref{fig:3Dattractor2}~(a) and its projections on the planes $(x_1,x_2)$, $(x_1,x_3)$ and $(x_2,x_3)$ in the Figures~\ref{fig:3Dattractor2}~(b), \ref{fig:3Dattractor2}~(c) and \ref{fig:3Dattractor2}~(d), respectively. Its largest Lyapunov exponent calculated is $\lambda=1.056$ (Figure~\ref{fig:3Dattractor2}~(e)).\\
\begin{figure}[h!]
	\centering
	(a)
	\includegraphics[width=0.4\columnwidth]{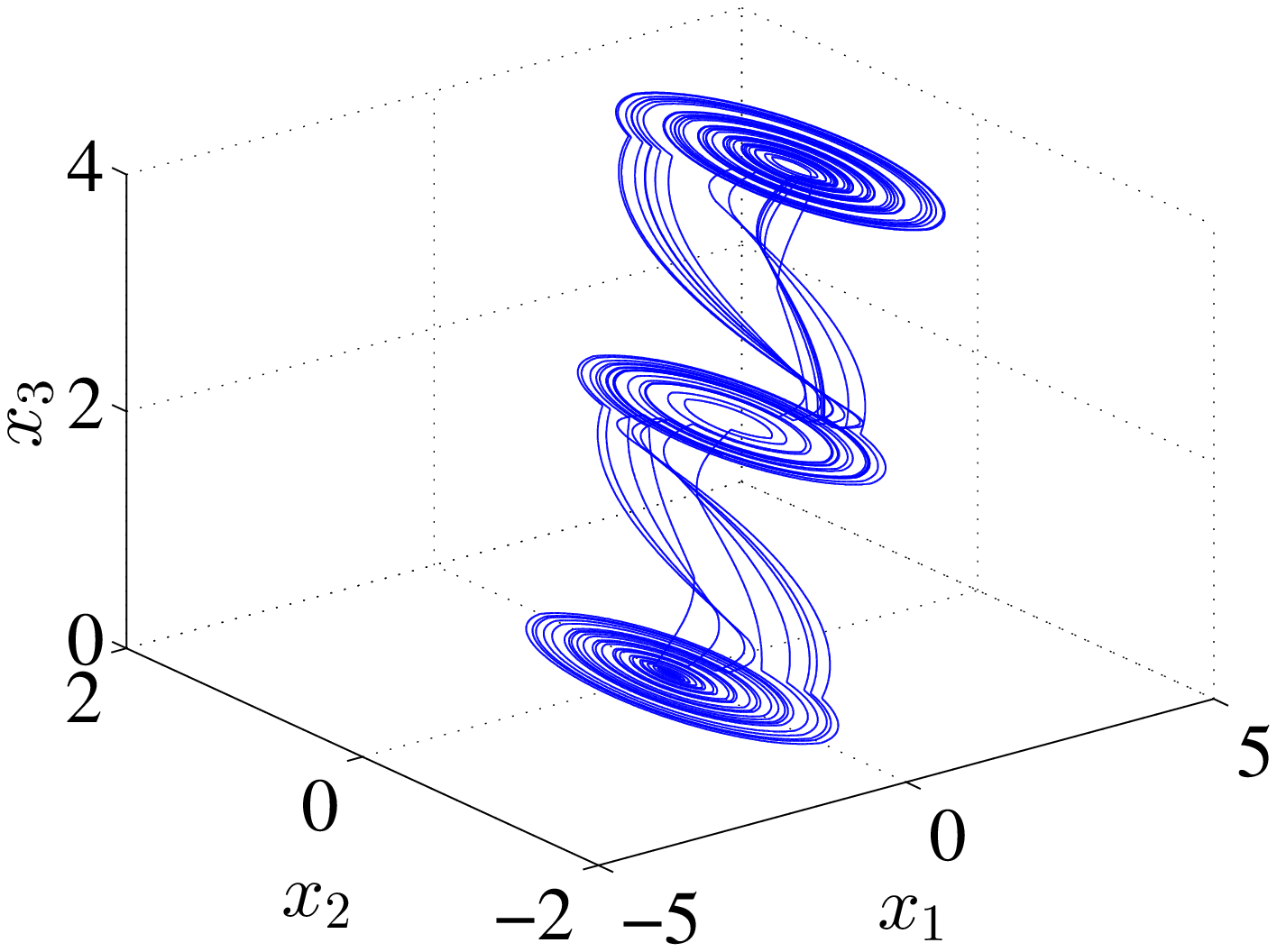}
	(b)
	\includegraphics[width=0.4\columnwidth]{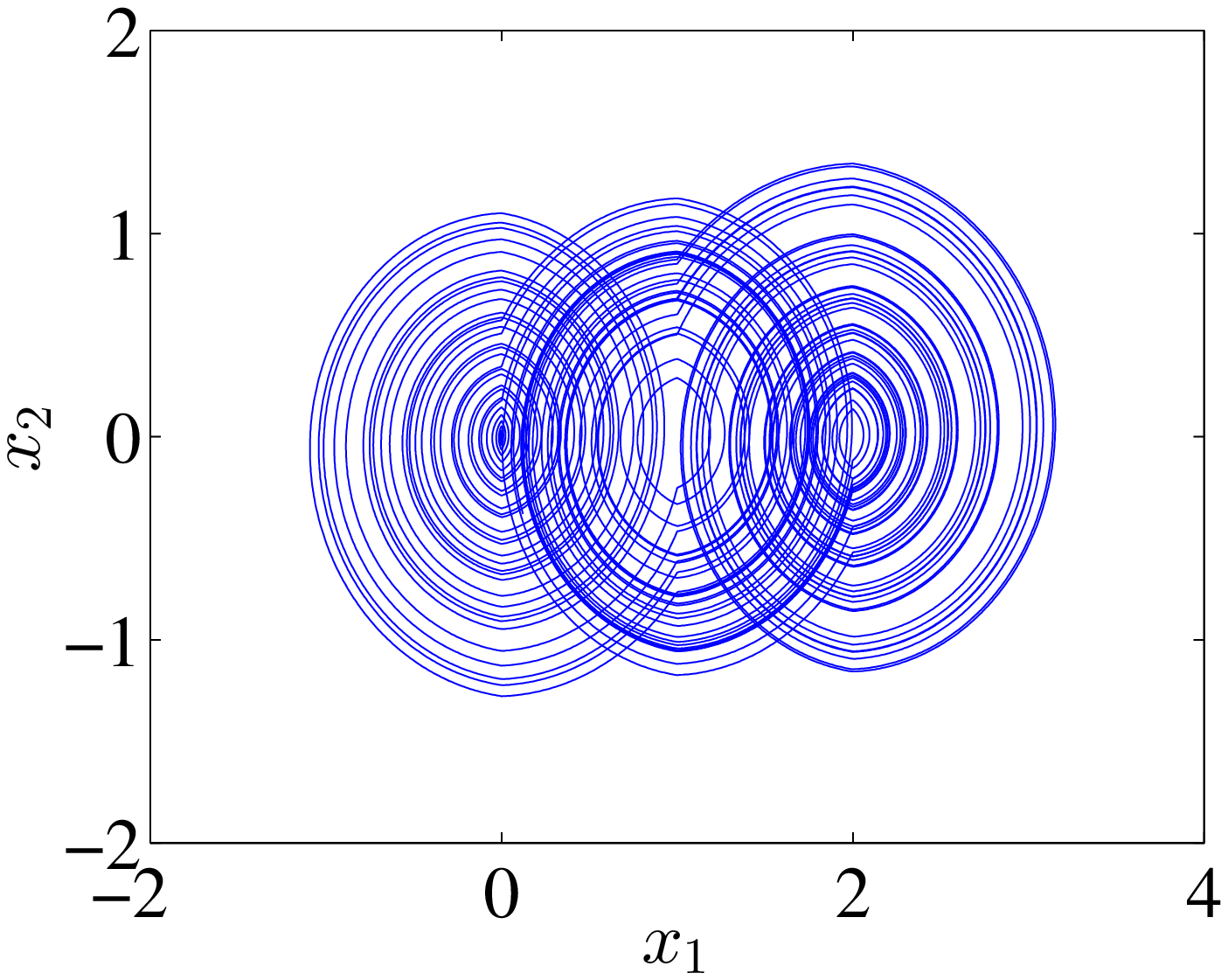}
	\\
	(c)
	\includegraphics[width=0.4\columnwidth]{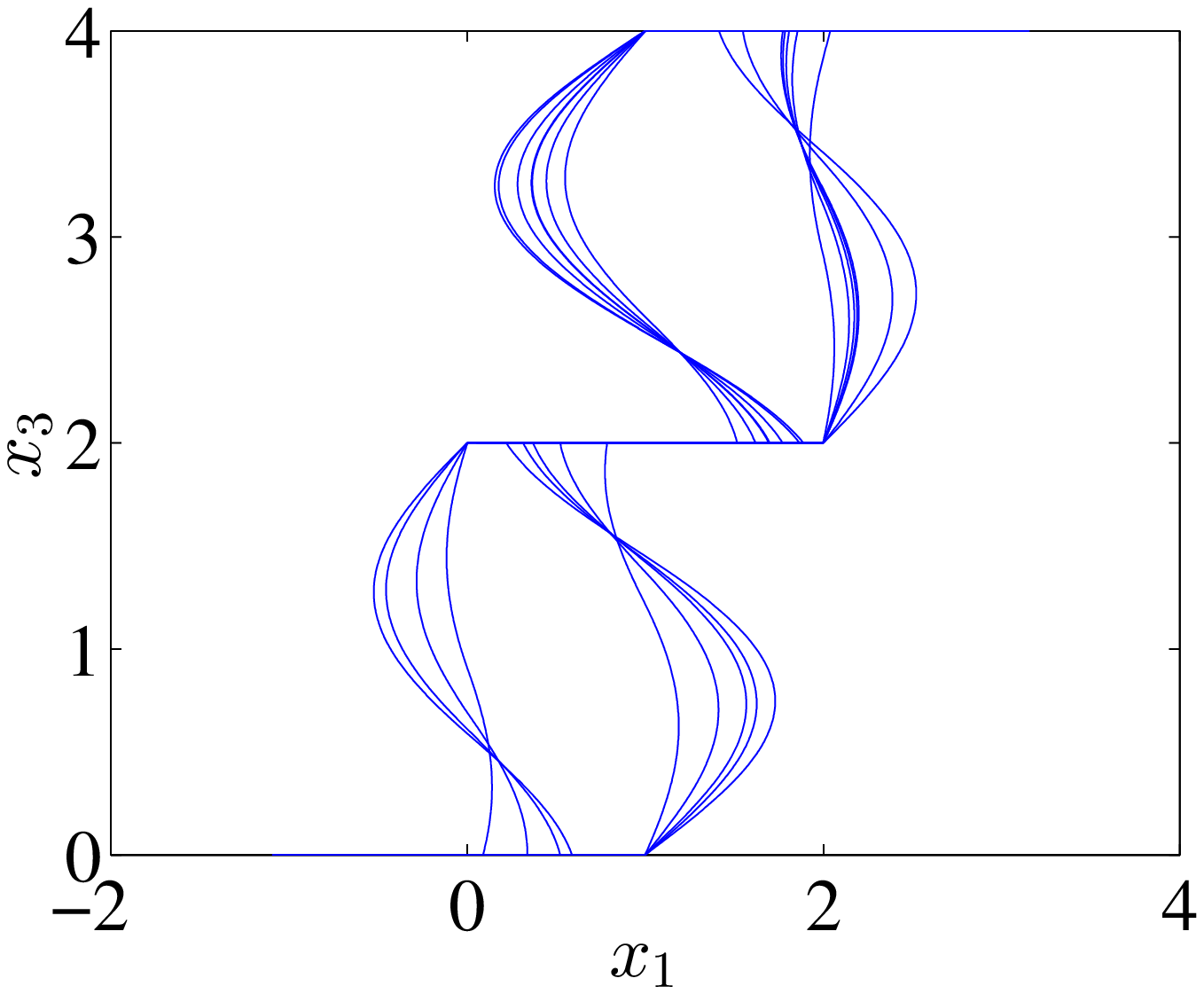}
	(d)
	\includegraphics[width=0.4\columnwidth]{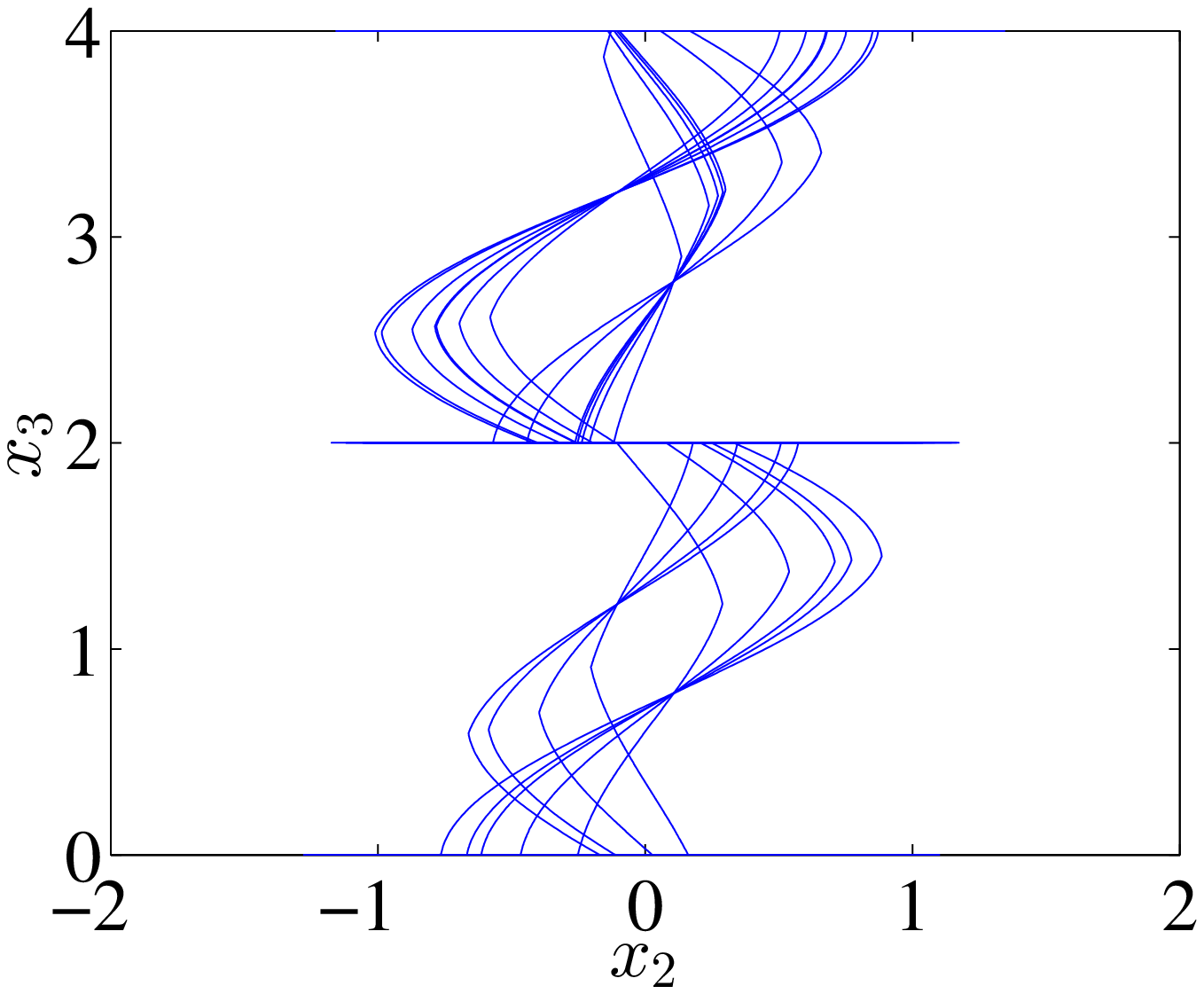}\\
	(e)
	\includegraphics[width=0.4\columnwidth]{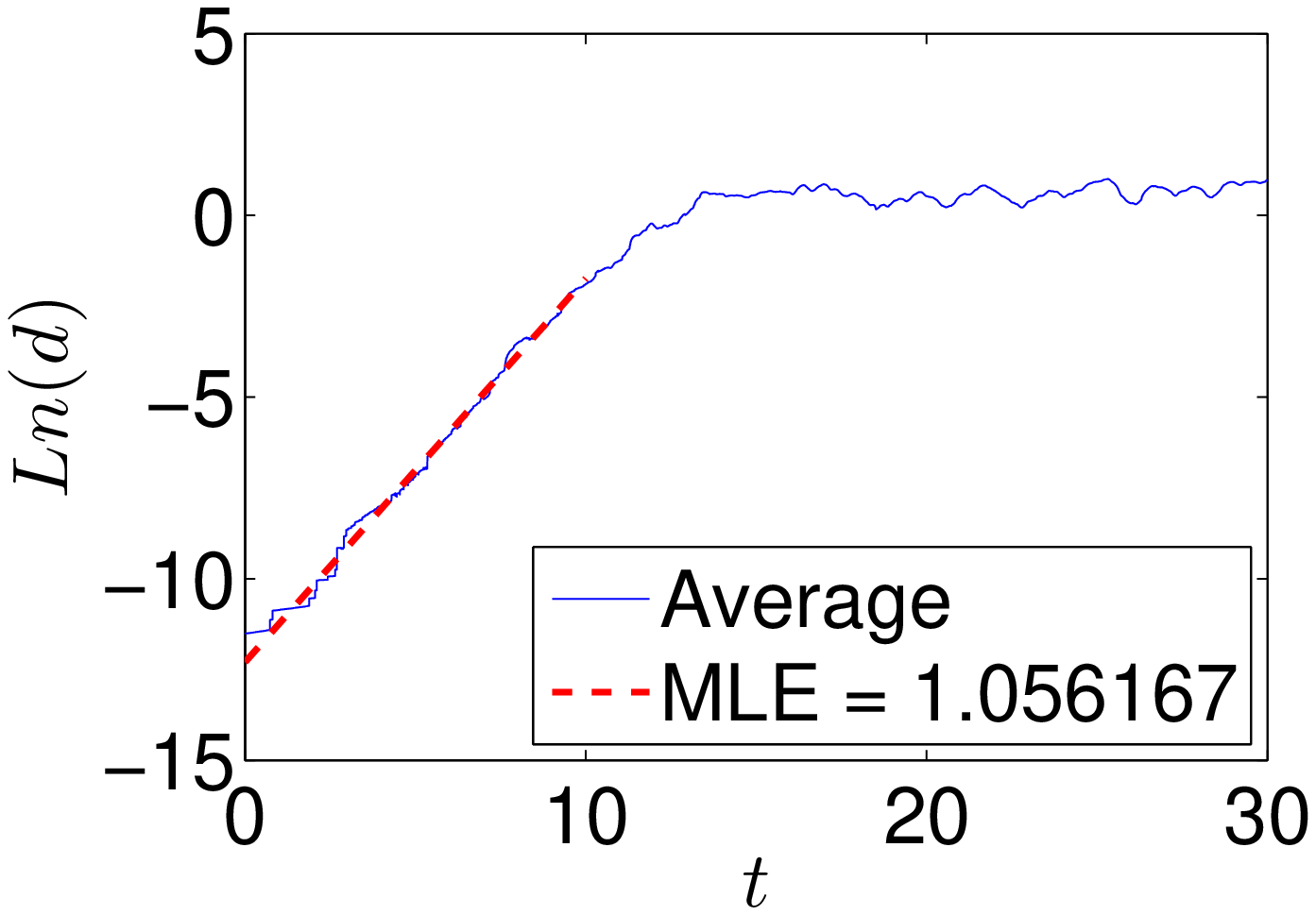}
	\caption{\label{fig:3Dattractor2}Triple scroll attractor for the initial condition $(0,0,0)$ in (a) the space $(x_1,x_2,x_3)$ and its projections onto the planes: (b) $(x_1,x_2)$, (c) $(x_1,x_3)$ and (d) $(x_2,x_3)$. In (e) the largest Lyapunov exponent.}
\end{figure}

Assigning the numbers 1, 3 and 5 to the regions where $\boldsymbol{x}<S_2$, $S_2\leq \boldsymbol{x}<S_4$, and $\boldsymbol{x}\geq S_4$, respectively, the symbol sequences 135, 131, 535 and 531 may be produced. Table~\ref{tab:example1seq} shows three sequences for different close initial conditions.\\

\begin{table}[h!]
	\centering
		\begin{tabular}{ll}
			$x(0)$&Sequence\\\hline
			(0.0,0.0,0.0)& 13135313135353531313535\\
			(0.0,0.1,0.0)& 131313535313135313135\\
			(-0.1,0,-0.1)& 13135353135313531353531\\\hline
		\end{tabular}
		\caption{\label{tab:example1seq}Sequences produced by the system from example 1 starting from three close to zero initial conditions for a 50s simulation with a RK4 of step=0.01.}
\end{table}

Later we show that this system passes the 0-1 Test for Chaos of Ref.~\cite{Gottwald}, justifying our description of chaotic motion.
\section{\label{sec:multiscrollconst2}Small perturbations of the zero eigenvalue: PWL dynamical system  $\dot{x}=Ax+B$ with multiscroll attractor and invertible $A$.}

Now we consider perturbing the real eigenvalue.   Consider the matrix $A_{\eta}$, where
\begin{equation}\label{eq:Amatrix2}
A=\begin{bmatrix}
m  & -n & 0\\
n 	 & m & 0\\
0 & 0  & \eta
\end{bmatrix},\hspace{.3cm}
A=[a_1,a_2,a_3],
\end{equation}
where $a_1$, $a_2$ and $a_3$ are the column vectors of the matrix $A$ and we suppose $m>0$ and $n\neq0$. The eigenvectors $V$ associated to the eigenvalue $\lambda=\eta$ are given as follows:

\begin{equation}
V=(0,0,v)^T,
\end{equation}
with $v\neq0$.\\

The column space of $A$  equals the two-dimensional unstable subspace $<~a_1,a_2~>$.  As before we consider
the vector field formed by  adding a vector $k_1 a_1+k_2 a_2$ in the span of the  column vectors. 

\begin{equation}
\boldsymbol{\dot{x}}=A\boldsymbol{x}+k_1a_1+k_2a_2.
\end{equation}
Using the matrix $A$ given by \eqref{eq:Amatrix2}, we have the following linear system:
\begin{equation}\label{ec:Aks2}
\boldsymbol{\dot{x}}=
\begin{bmatrix}
m& -n & 0\\
n& m & 0\\
0& 0 & \eta
\end{bmatrix}
\begin{bmatrix}
x_1+k_1\\x_2+k_2\\x_3
\end{bmatrix}.
\end{equation}
As before the  solution of the initial value problem  is given by
\begin{equation}\label{ec:solAx2}
\boldsymbol{x(t)}=
\begin{bmatrix}
e^{mt}cos(nt)& -e^{mt}sin(nt) & 0\\
e^{mt}sin(nt)& e^{mt}cos(nt) & 0\\
0& 0 & e^{\eta t}
\end{bmatrix}
\begin{bmatrix}
x_1(0)+k_1\\x_2(0)+k_2\\x_3(0)
\end{bmatrix}+
\begin{bmatrix}
-k_1\\-k_2\\\frac{v}{\eta}(e^{\eta t}-1)
\end{bmatrix}.\end{equation}

Consider the equation $\dot{x}=\eta x+w$ which has solution $x(t)=e^{\eta t } [x (0) + \frac{w}{\eta}] -\frac{w}{\eta}$. 
Suppose that the sign of $v$ is such  that the flow is directed toward $S_1$ in Figure 1. So $\dot{x}_3=\eta x_3-|v|$ if
$x_3(0)>0$ and $\dot{x}_3=\eta x_3+|v|$ if $x_3 (0)<0$.
If $\eta <0$ and $x_3 (0)>0$ then $\lim_{t \to \infty} x_3 (t)=\frac{|v|}{\eta}$, so that the flow in Figure 1 still intersects
$S_1$ from an initial condition $x_3 (0)>0$. Similarly  if  $x_3 (0)<0$ then the flow intersects $S_1$. Thus the topological structure of the flow is unchanged for small $|\eta|$, $\eta <0$. 

However if $\eta >0$ then $\lim_{t\to \infty} x_3(t)\to \infty$ if $x_3(0)>\frac{|v|}{\eta}$ and $\lim_{t\to \infty} x_3(t)\to - \infty$ for $x_3(0)<-\frac{|v|}{\eta}$. So in this case there is a neighborhood of $S_1$ of
points closer than $|\frac{|v|}{\eta}|$ to $S_1$ consisting of points attracted to $S_1$. Otherwise points are repelled from $S_1$.

Consider now the solution  where the sign of $v$ is such  that the flow is directed away from  $S_1$ in Figure 1. So $\dot{x}_3=\eta x_3+|v|$ if
$x_3(0)>0$ and $\dot{x}_3=\eta x_3-|v|$ if $x_3 (0)<0$.

Clearly if $\eta >0$ and $x_3 (0)>0$ then $\lim_{t \to \infty} x_3 (t)=\infty$ and if $\eta >0$ and $x_3 (0)<0$ , then $\lim_{t \to \infty} x_3 (t)=-\infty$. However if $\eta<0$ and $x_3 (0)>0$ then 
$\lim_{t \to \infty} x_3 (t)=\frac{-|v|}{\eta}$ and if $\eta<0$ and $x_3 (0)<0$ then $\lim_{t \to \infty} x_3 (t)=\frac{|v|}{\eta}$. 
Thus for small $|\eta|$ the multi-scroll attractors persist. 

\begin{figure}[h!]
	\centering
	(a)
	\includegraphics[width=0.4\columnwidth]{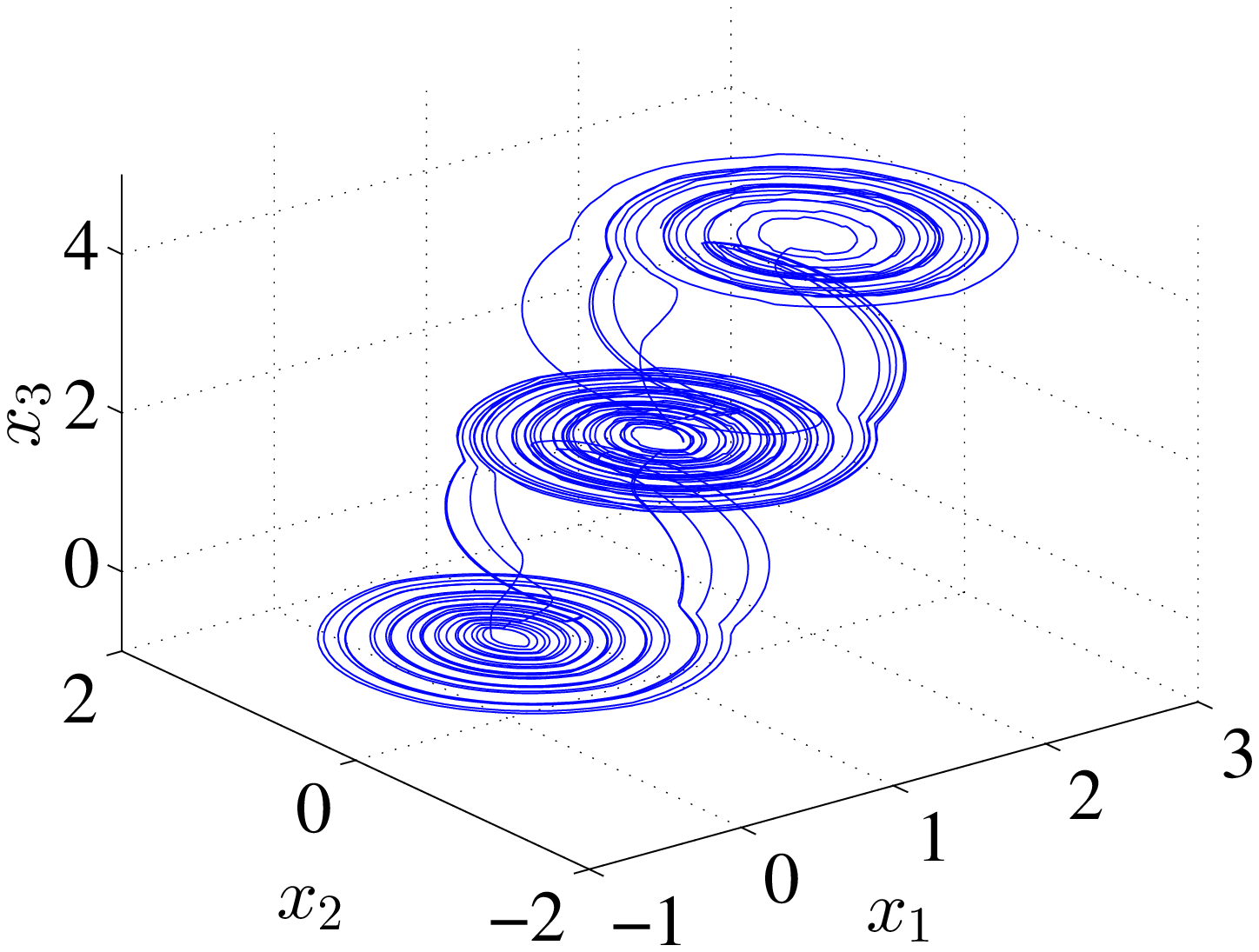}
	(b)
	\includegraphics[width=0.4\columnwidth]{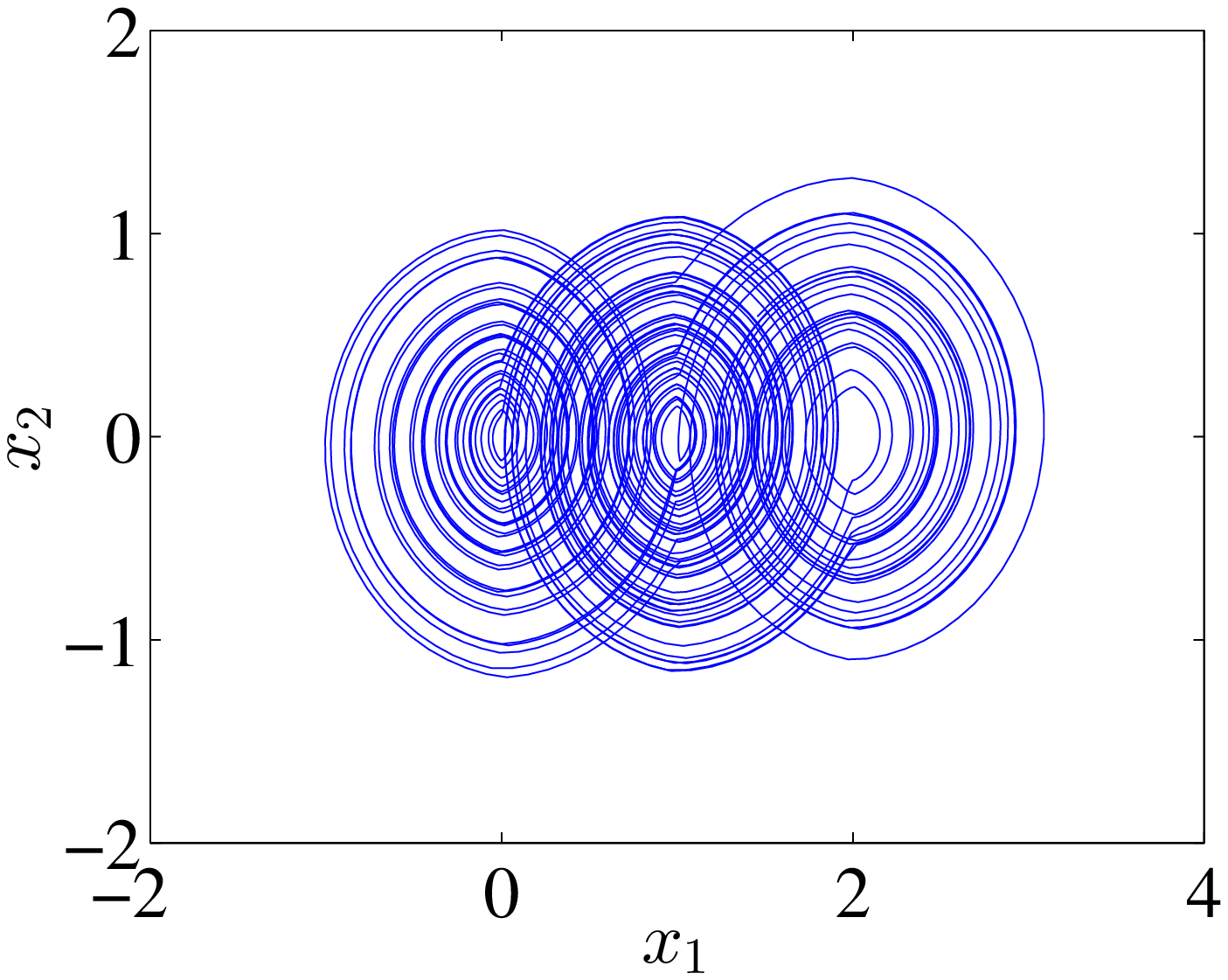}
	\\
	(c)
	\includegraphics[width=0.4\columnwidth]{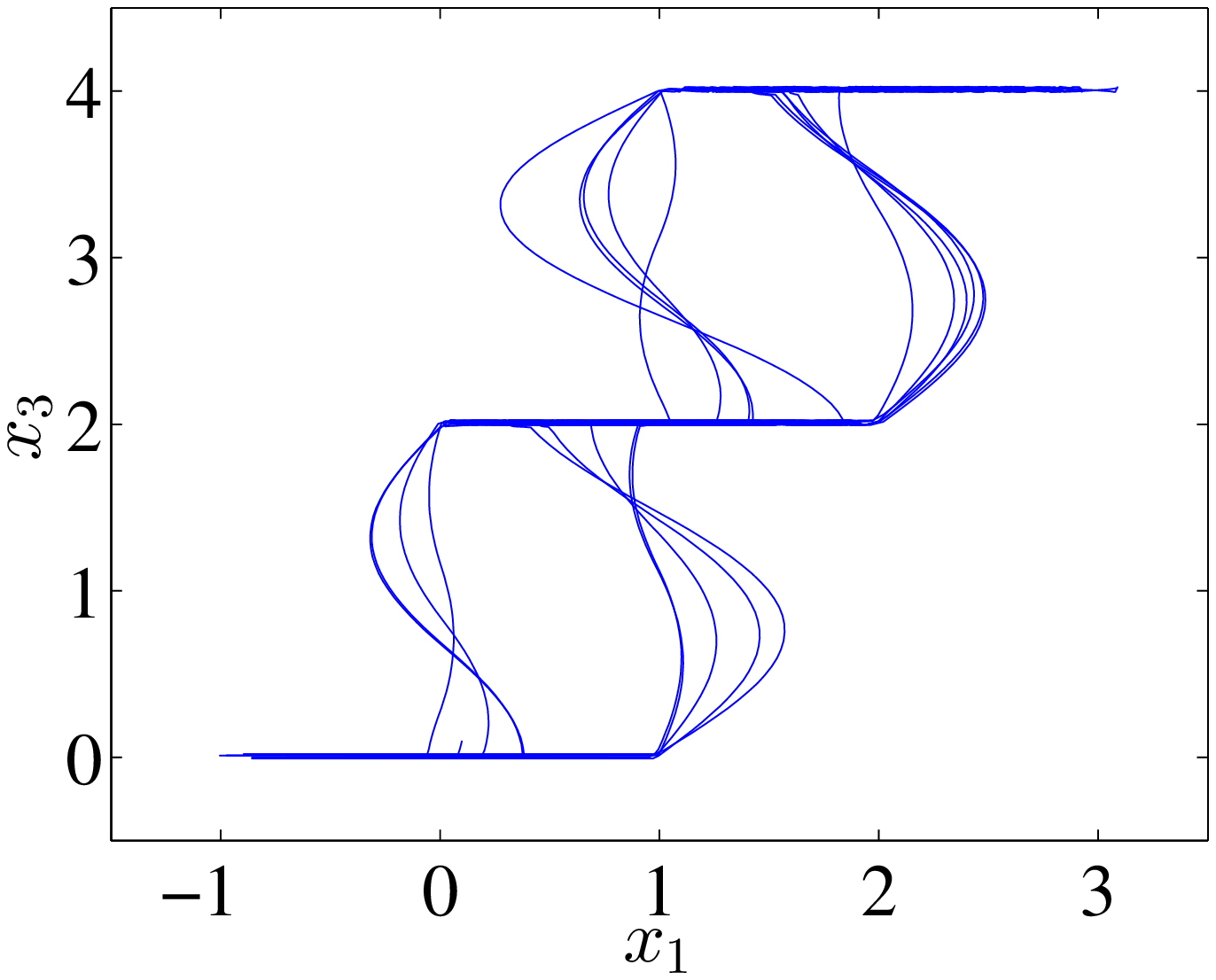}
	(d)
	\includegraphics[width=0.4\columnwidth]{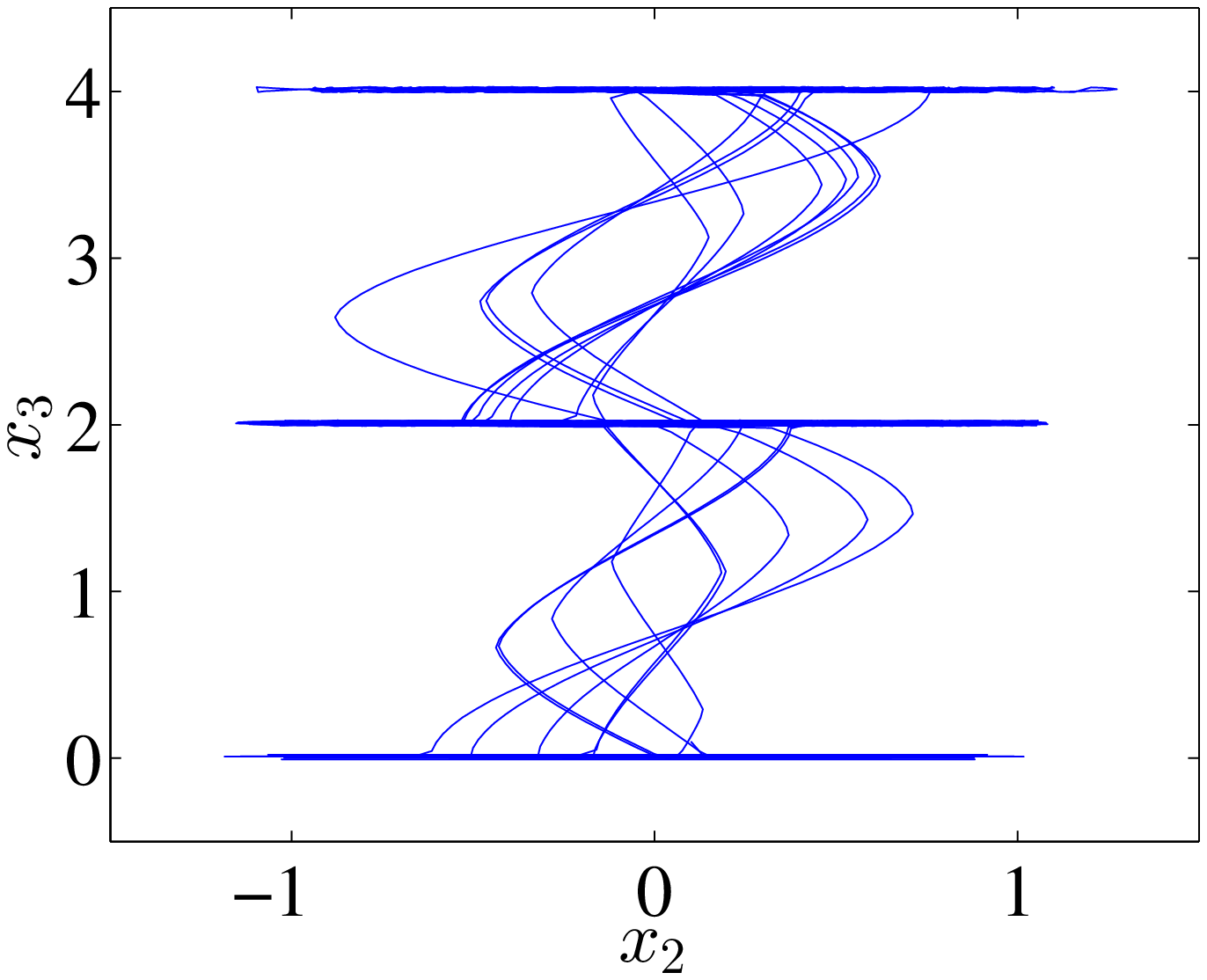}\\(e)
	\includegraphics[width=0.4\columnwidth]{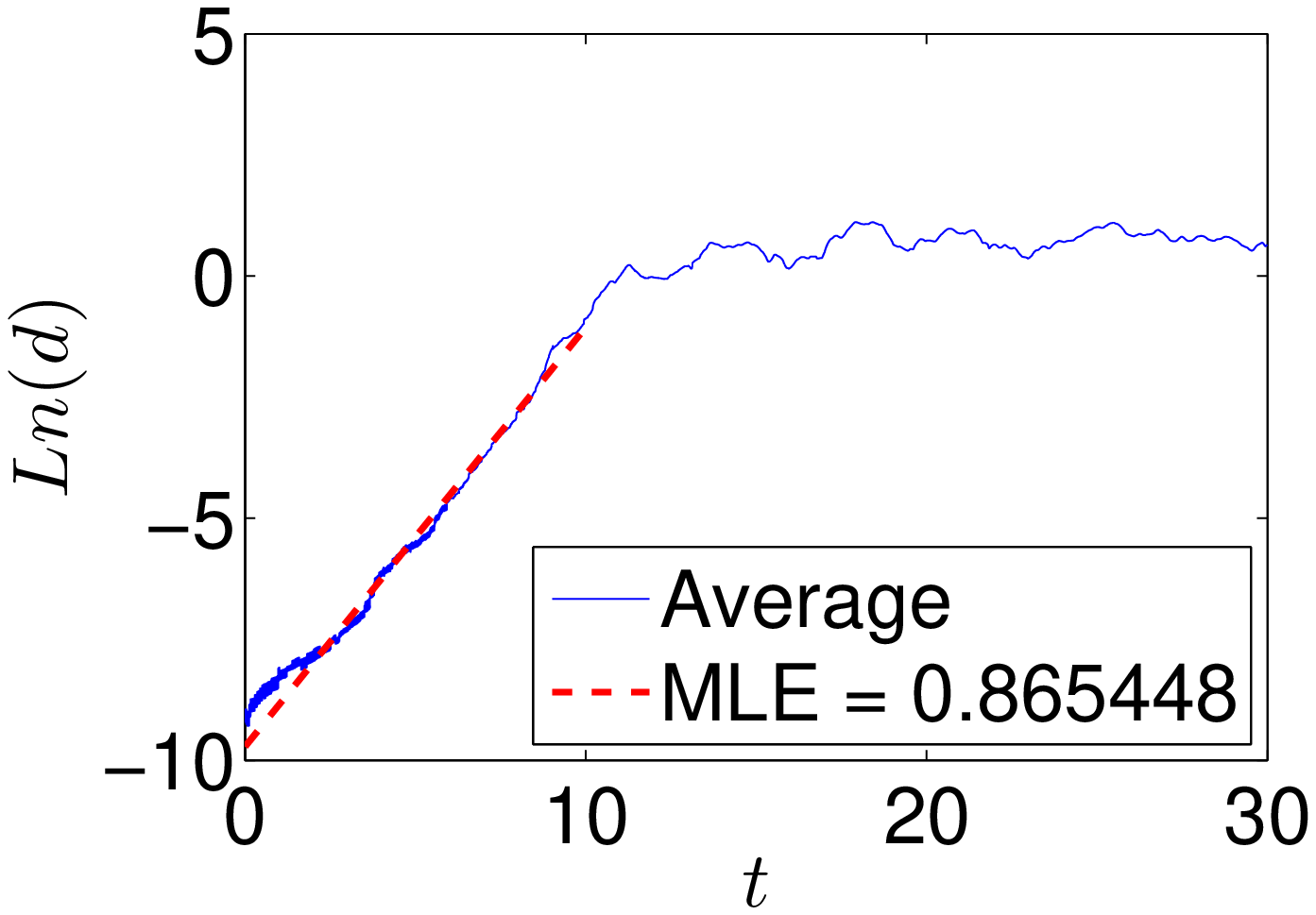}
	\caption{Attractor of the system (\ref{eq:invertible3scroll}) with $A$ given in (\ref{eq:Amatrixex2}) for the initial condition $(0.1,0.1,0.1)$ in (a) the space $(x_1,x_2,x_3)$ and its projections onto the planes: (b) $(x_1,x_2)$, (c) $(x_1,x_3)$ and (d) $(x_2,x_3)$. In (e) the largest Lyapunov exponent.}\label{fig_invertible3scroll}
\end{figure}

\subsection{Numerical simulations}
{\bfseries Example 2. }As an example of this construction with an invertible matrix, consider the system described by:

\begin{equation}\label{eq:invertible3scroll}
\boldsymbol{\dot{x}}= F_i(\boldsymbol{x}), i=1,\ldots,18.
\end{equation}
Detailes of the function \eqref{eq:invertible3scroll} are in Appendix II. The linear operator and  the the vector $V$ are given as follows
\begin{equation}\label{eq:Amatrixex2}
A=
\begin{bmatrix}
0.5& -10 & 0\\
10& 0.5 & 0\\
0& 0 & 0.1
\end{bmatrix},
V=\begin{bmatrix}
0\\
0\\
5
\end{bmatrix}.
\end{equation}
\\
The switching surfaces are given by the planes $S_1: x_3=0, \ S_2:x_1+x_3/2=1, \ S_3:x_3=2, \ S_4:x_1+x_3/2=3, \ S_5:x_3=4$. The vectors $W_i$, with $i=1,\ldots,6$ are given in Table~\ref{tab:example1} (See Appendix I).

The resulting attractor obtained by using a 4th order Runge Kutta (0.01 integration step) is shown in Figure~\ref{fig_invertible3scroll}.\\
Assigning the numbers 1, 3 and 5 to the regions where $\boldsymbol{x}<S_2$, $S_2\leq \boldsymbol{x}<S_4$, and $\boldsymbol{x}\geq S_4$, respectively, the symbol sequences 135, 131, 535 and 531 may be produced. Table~\ref{tab:example2seq} shows three sequences generated from different close to zero initial conditions. \\

\begin{table}[h!]
	\centering
		\begin{tabular}{ll}
			$x(0)$&Sequence\\\hline
			(0.1,0.1,0.1)&123232121232121232323\\
			(0.0,0.1,0.0)&1212321232121212123\\
			(-0.1,0,-0.1)&123212121212321232321\\\hline
		\end{tabular}
		\caption{\label{tab:example2seq}Sequences produced by the system from example 2 starting from three close to zero initial conditions for a 50s simulation with a RK4 of step=0.01.}
\end{table}

%
As for the previous system, we  show in  the next section  that this system passes the 0-1 Test for Chaos of Ref.~\cite{Gottwald}, justifying our description of chaotic motion.
\section{\label{sec:dynamics}Dynamics of the proposed systems}

To test for chaotic  dynamics in the systems we have investigated we used  the test algorithm proposed in Ref.~\cite{Gottwald}. The input for the test is a one dimensional time series $\phi(n)$ which drives a 2-dimensional system $PQ(\phi(n),c)$ as described in Ref.~\cite{Gottwald}, namely

\[
p_c(n)=\sum_{j=1}^n\phi (j) \cos (jc)
\]
\[
q_c(n)=\sum_{j=1}^n\phi (j) \sin (jc)
\]
where $c\in (0,2\pi)$ is a real parameter. The rate of growth of  the variance of this system distinguishes between chaotic  ($K=1$) and regular motion  $(K=0)$ as determined by a derived quantity $K$.

For both systems with a triple scroll attractor previously introduced, 3-dimensional time series were generated by means of a RK4 integrator with a time step equal to $0.01$ which was then sampled T time $\tau=0.25$ to get 3-dimensional time series of length $N=2000$. The 1-dimensional time-series was given by $\phi \circ T^n (p_0)$ for an initial condition $p_0=(x_1^0, x_2^0, x_3^0)$
and $\phi (x_1,x_2,x_3)=x_3$. So that we were observing the $z$-component of a trajectory  under the time -T map of the flow.

The growth rates calculated were $K= 0.9930$ and $K=0.9962$ for the first and second example, respectively. As described in Ref.~\cite{Gottwald} $K$ is calculated as the median value of the asymptotic growth rates of a growth rate $K_c$ for different values of the parameter $c$ belonging to the 2-dimensional driven system. In figures~\ref{fig:kcsys}(a) and \ref{fig:kcsys}(b) the $K_c$ values are shown.
Thus our system, according to  the 0-1 test, is chaotic.

\begin{figure}[h!]
	\centering
	(a)
	\includegraphics[width=0.4\columnwidth]{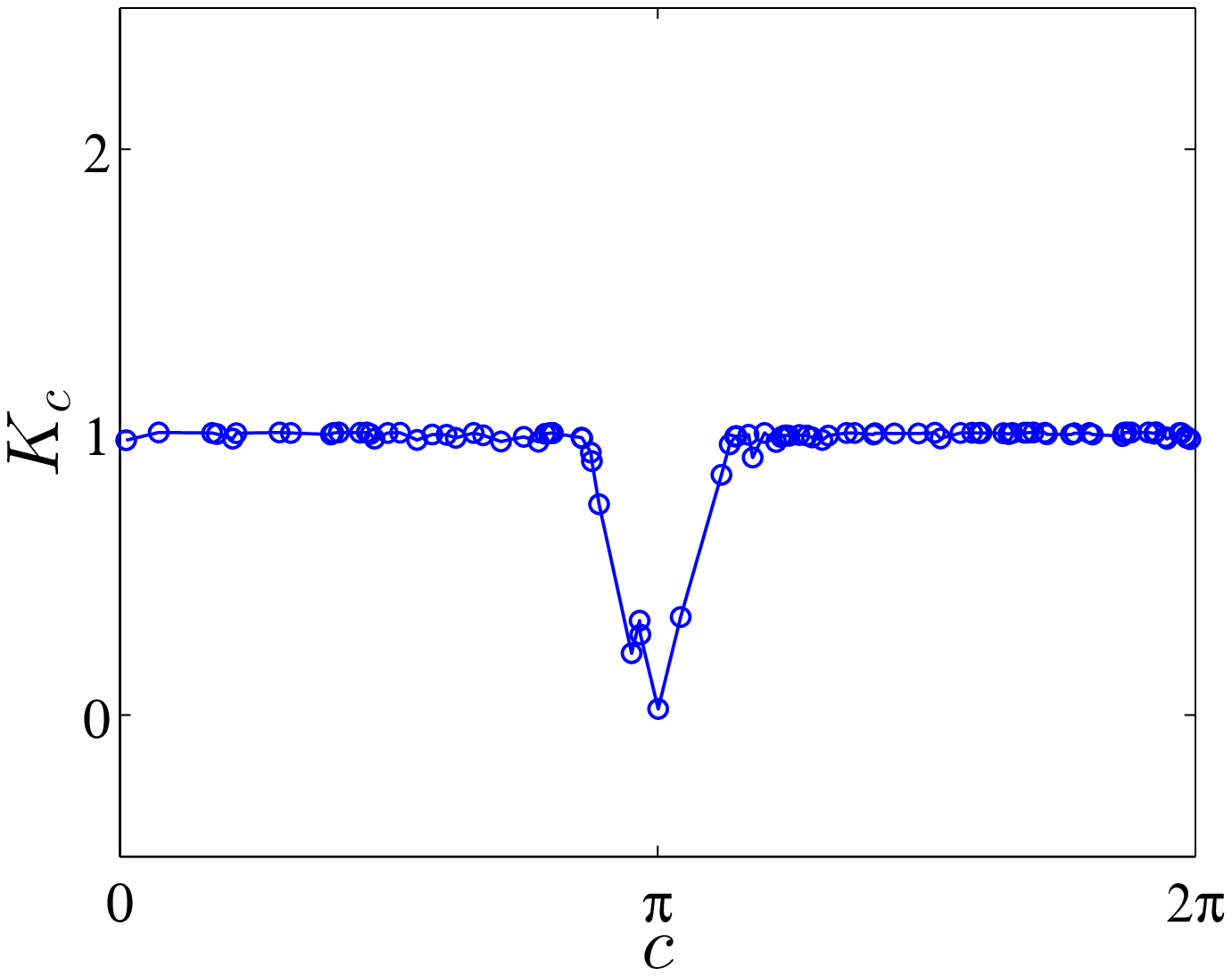}
	(b)
	\includegraphics[width=0.4\columnwidth]{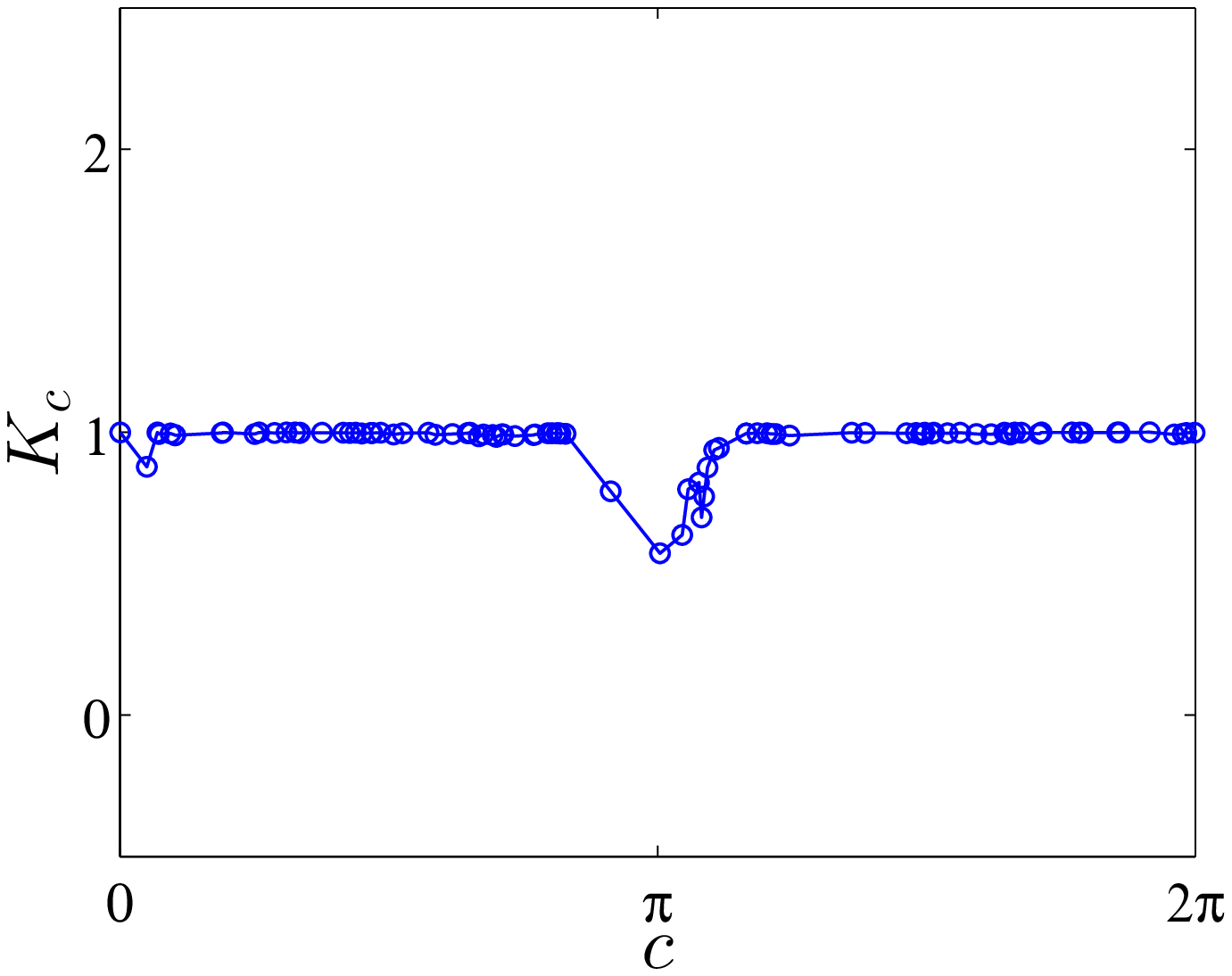}
	\caption{\label{fig:kcsys} Asymptotic growth rates $Kc$ calculated for the triple scroll attractor systems presented in sections (a)\ref{sec:multiscrollconst} and (b)\ref{sec:multiscrollconst2}.}
\end{figure}

\section{Conclusion}

In this paper new classes of piecewise linear dynamical systems (with no equilibria) which present multiscroll attractors were introduced. The systems are very simple geometrically and 
stable to perturbation. The attractors generated by these classes have a similar chaotic  behavior. Further investigation, via perhaps symbolic dynamics, is needed to quantify the disorder of the system.

\section*{Acknowledgment}
R.J. Escalante-Gonz\'alez PhD student of control and dynamical systems at IPICYT thanks CONACYT(Mexico) for the scholarship granted (Register number 337188). E. Campos-Cant\'on acknowledges CONACYT for the financial support for sabbatical year. M. Nicol  thanks the NSF for partial support on NSF-DMS Grant 1600780.

\section*{ Appendix I}

System described by \eqref{eq:2SPWLS}:
\begin{equation}\label{}
\boldsymbol{\dot{x}}=
\begin{cases}
A\boldsymbol{x}+V+W_1, & \text{if } \boldsymbol{x} <S_1, x_1< 0;\\
A\boldsymbol{x}+V+W_2, & \text{if } \boldsymbol{x} <S_1,\boldsymbol{x} <S_2, x_1 \ge 0;\\
A\boldsymbol{x}+W_1, &  \text{if } \boldsymbol{x}\in S_1, x_1< 0;\\
A\boldsymbol{x}+W_2, &  \text{if } \boldsymbol{x}\in S_1,\boldsymbol{x} <S_2, x_1 \ge 0;\\
A\boldsymbol{x}-V+W_1, & \text{if } \boldsymbol{x} > S_1,\boldsymbol{x} <S_2, x_1< 0 ;\\
A\boldsymbol{x}-V+W_2, & \text{if } \boldsymbol{x} > S_1,\boldsymbol{x} <S_2, x_1 \ge 0;\\
A\boldsymbol{x}+V+W_3, & \text{if } \boldsymbol{x} <S_3,\boldsymbol{x} \geq S_2, x_1< 1;\\
A\boldsymbol{x}+V+W_4, & \text{if } \boldsymbol{x} <S_3, \boldsymbol{x} \geq S_2, x_1 \ge 1;\\
A\boldsymbol{x}+W_3, & \text{if } \boldsymbol{x}\in S_3,\boldsymbol{x} \geq S_2, x_1< 1;\\
A\boldsymbol{x}+W_4, & \text{if } \boldsymbol{x}\in S_3, x_1 \ge 1;\\
A\boldsymbol{x}-V+W_3, & \text{if } \boldsymbol{x} > S_3,\boldsymbol{x} \geq S_2, x_1< 1 ;\\
A\boldsymbol{x}-V+W_4, & \text{if } \boldsymbol{x} > S_3, x_1 \ge 1
\end{cases}
\end{equation}

The switching surface are given by the planes $S_1: x_3=0$, $S_2: x_1+x_3/2=1$ and $S_3: x_3=2$ .
The vectors $W_i$, with $i=1,\ldots,4$ are given in Table~\ref{tab:example1}\\

\begin{table}[h!]
	\centering
		\begin{tabular}{ccc}
			$W_i$&  $k_1$ & $k_2$\\
			\hline
			$W_1$  &  -0.1 & 0\\
			$W_2$ &  0.1 & 0\\
			$W_3$ &  -1.1 & 0\\
			$W_4$ &  -0.9 & 0\\
			$W_5$ &  -2.1 & 0\\
			$W_6$ &  -1.9 & 0\\\hline
		\end{tabular}
		\caption{\label{tab:example1}Vectors $W_i$ for $i=1,\ldots,6$.}
\end{table}

\section*{Appendix II}
System described by \eqref{eq:invertible3scroll}:

\begin{equation}
\boldsymbol{\dot{x}}=
\begin{cases}
A\boldsymbol{x}+V+W_1, & \text{if } \boldsymbol{x} <S_1, x_1< 0;\\
A\boldsymbol{x}+V+W_2, & \text{if } \boldsymbol{x} <S_1,\boldsymbol{x} <S_2, x_1 \ge 0;\\
A\boldsymbol{x}+W_1,   & \text{if } \boldsymbol{x}\in S_1, x_1< 0;\\
A\boldsymbol{x}+W_2,   & \text{if } \boldsymbol{x}\in S_1,\boldsymbol{x} <S_2, x_1 \ge 0;\\
A\boldsymbol{x}-V+W_1, & \text{if } \boldsymbol{x} > S_1,\boldsymbol{x} <S_2, x_1< 0 ;\\
A\boldsymbol{x}-V+W_2, & \text{if } \boldsymbol{x} > S_1,\boldsymbol{x} <S_2, x_1 \ge 0;\\
A\boldsymbol{x}+V+W_3, & \text{if } \boldsymbol{x} <S_3,\boldsymbol{x} \geq S_2, x_1< 1;\\
A\boldsymbol{x}+V+W_4, & \text{if } \boldsymbol{x} <S_3, \boldsymbol{x} \geq S_2, \boldsymbol{x} < S_4, x_1 \ge 1;\\
A\boldsymbol{x}+W_3,   & \text{if } \boldsymbol{x}\in S_3,\boldsymbol{x} \geq S_2,  x_1< 1;\\
A\boldsymbol{x}+W_4,   & \text{if } \boldsymbol{x}\in S_3,\boldsymbol{x} < S_4 ,x_1 \ge 1;\\ 
A\boldsymbol{x}-V+W_3, & \text{if } \boldsymbol{x} > S_3,\boldsymbol{x} \geq S_2, \boldsymbol{x} < S_4, x_1< 1 ;\\ 
A\boldsymbol{x}-V+W_4, & \text{if } \boldsymbol{x} > S_3, \boldsymbol{x} < S_4, x_1 \ge 1\\
A\boldsymbol{x}+V+W_5, & \text{if } \boldsymbol{x} <S_5,\boldsymbol{x} \geq S_4, x_1< 2;\\
A\boldsymbol{x}+V+W_6, & \text{if } \boldsymbol{x} <S_5, \boldsymbol{x} \geq S_4, x_1 \ge 2;\\
A\boldsymbol{x}+W_5,   & \text{if } \boldsymbol{x}\in S_5,\boldsymbol{x} \geq S_4,  x_1< 2;\\
A\boldsymbol{x}+W_6,   & \text{if } \boldsymbol{x}\in S_5,x_1 \ge 2;\\ 
A\boldsymbol{x}-V+W_5, & \text{if } \boldsymbol{x} > S_5,\boldsymbol{x} \geq S_4, x_1< 2 ;\\ 
A\boldsymbol{x}-V+W_6, & \text{if } \boldsymbol{x} > S_5, x_1 \ge 2\\

\end{cases}
\end{equation}

\end{document}